\documentclass{article}
\usepackage[a4paper, total={6in, 8in}]{geometry}
\usepackage[english]{babel}
\usepackage{amssymb}
\usepackage{amsthm}
\usepackage{amsmath}
\usepackage{graphicx}
\usepackage{array,hhline}
\usepackage{enumerate}
\newtheorem{lemma}{Lemma}[section]
\newtheorem{theorem}[lemma]{Theorem}
\newtheorem{assumption}[lemma]{Assumption}
\newtheorem{proposition}[lemma]{Proposition}
\newtheorem{conjecture}[lemma]{Conjecture}
\newtheorem{corollary}[lemma]{Corollary}
\theoremstyle{definition}
\newtheorem{definition}[lemma]{Definition}

\numberwithin{equation}{section}
\numberwithin{figure}{section}

\newcommand{\nix}{{\mbox{\vphantom x}}}

\newcommand{\Aset}{\mathcal{A}}
\newcommand{\Bset}{\mathcal{B}}
\newcommand{\Cset}{\mathcal{C}}

\newcommand{\Lset}{\mathcal{L}}

\newcommand{\Sset}{\mathcal{S}}

\newcommand{\Xset}{\mathcal{X}}

\newcommand{\md}{$\widetilde{\ell}$-m}

\newcommand{\wti}[1]{{\widetilde #1}}

\newcommand{\upperRomannumeral}[1]{\uppercase\expandafter{\romannumeral#1}}

\title{A new proof of the Gasca - Maeztu conjecture for n = 5}
\author{Gagik Vardanyan}
\date{July 2021}

\begin{document}

\maketitle
\begin{abstract}  An $n$-correct node set $\mathcal{X}$ is called $GC_n$ set if the fundamental polynomial of each node is a product of $n$ linear factors. In 1982 Gasca and Maeztu conjectured that for every $GC_n$ set there is a line passing through $n+1$ of its nodes.So far, this conjecture has been confirmed only for $n\le 5.$ The case $n = 4,$ was first proved by J. R. Bush in 1990 . Several other proofs have been published since then. For the case $n=5$ there is only one proof: by H. Hakopian, K. Jetter and G. Zimmermann (Numer Math $127,685--713, 2014$). Here we present a second, much shorter and easier proof.
\end{abstract}



\section{Introduction}
Denote by $\Pi_n$ the space of bivariate polynomials of total degree
at most $n:$
\begin{equation*}
\Pi_n=\left\{\sum_{i+j\leq{n}}a_{ij}x^iy^j
\right\},\quad N:=\dim \Pi_n=\binom{n+2}{2}.
\end{equation*}

\noindent Consider a set of distinct nodes
$\Xset_s=\{ (x_1, y_1), (x_2, y_2), \dots , (x_s, y_s) \} .
$

The problem of finding a polynomial $p \in \Pi_n$ which satisfies
the conditions
\begin{equation}\label{int cond}
p(x_i, y_i) = c_i, \ \ \quad i = 1, 2, \dots s  ,
\end{equation}
is called interpolation problem.
\begin{definition}
The interpolation problem with the set of nodes $\Xset_s$ is called
$n$-poised if for any data $\{c_1, \dots, c_s\}$ there exists a
unique polynomial $p \in \Pi_n$, satisfying the conditions
\eqref{int cond}.
\end{definition}

A necessary condition of
$n$-poisedness is: $\#\Xset_s=s = N.$ If this latter equality takes place then the following holds:

\begin{proposition} \label{poisedii}
A set of nodes $\Xset_N$ is $n$-poised if and only if 
$$p \in \Pi_n,\ p(x_i,
y_i) = 0 \quad i = 1, \dots , N \implies p = 0.$$
\end{proposition}
A polynomial $p \in \Pi_n$ is called an $n$-fundamental polynomial
for a node $ A = (x_k, y_k) \in \Xset_s$ if
\begin{equation*}
p(x_i, y_i) = \delta _{i k},\  i = 1, \dots , s ,
\end{equation*}
where $\delta$ is the Kronecker symbol. We denote the
$n$-fundamental polynomial of $A \in\Xset_s$ by $p_A^\star=p_{A,\Xset}^\star.$

\begin{definition}
A set of nodes $\Xset_s$ is called $n$-independent if all its nodes
have $n$-fundamental polynomials. Otherwise, $\Xset_s$ is called
$n$-dependent.
A set of nodes $\Xset_s$ is called essentially $n$-dependent if none of its nodes
has $n$-fundamental polynomial.
\end{definition}
Fundamental polynomials are linearly independent. Therefore a
necessary condition of $n$-indepen-dence is $\#\Xset_s=s \le N.$ 

One can readily verify that a node set $\Xset_s$
is $n$-independent if and only if the interpolation problem
\eqref{int cond} is solvable, meaning that for any data $\{c_1, \dots , c_s
\}$ there exists a (not necessarily unique) polynomial $p \in \Pi_n$
satisfying the conditions \eqref{int cond}.

A plane algebraic curve is the zero set of some bivariate polynomial of degree $\ge 1.$~To simplify notation, we shall use the same letter,  say $p$,
to denote the polynomial $p$ of degree $\ge 1$ and the curve given by the equation $p(x,y)=0$.
In particular, by $\ell,$ we denote a linear 
polynomial $\ell\in\Pi_1$ and the line defined by the equation
$\ell(x, y)=0.$

\begin{definition} Let $\Xset$ be an $n$-poised set.
We say, that a node $A\in\Xset$
uses a line $\ell$, if $\ell$ is a factor of the fundamental
polynomial $p_{A}^\star,$ i.e., $p_{A}^\star=\ell q,$ where $q\in\Pi_{n-1}.$
\end{definition}

Since the fundamental polynomial of a node in an $n$-poised set is unique we
get
\begin{lemma}[\cite{HJZ09b},  Lemma 2.5] \label{lm}
Suppose $\Xset$ is a poised set and a node  $A\in \Xset$ uses a line $\ell.$ Then $\ell$ passes through at least two nodes from $\Xset$, at which $q$ does not vanish.
 \end{lemma}

\begin{definition} Let $\Xset$ be a set of nodes.
We say, that a line $\ell$ is a $k$-node line if it passes through exactly $k$ nodes of $\Xset:\ \ell\cap\Xset=k.$\end{definition}

The following proposition is well-known (see e.g. \cite{HJZ09a}
Proposition 1.3):
\begin{proposition}\label{pointsell}
Suppose that a polynomial $p \in
\Pi_n$ vanishes at $n+1$ points of a line $\ell.$ Then we have that
$
p = \ell  r  ,\ \text{where} \ r\in\Pi_{n-1}.
$
\end{proposition}
From here we readily get that at most $n+1$ nodes of an $n$-poised
set $\Xset_N$ can be collinear. In view
of this an $(n+1)$-node line $\ell$ is called a maximal line \cite{dB07}. \\

Next, let us bring the Cayley-Bacharach theorem (see e.g. \cite{E96}, Th. CB4;    
 \cite{HJZ09a}, Prop.~4.1).

\begin{theorem}\label{thm:C-B}
Assume that two algebraic curves of degree $m$ and $n$, respectively,
intersect at $m n$ distinct points. Then the set $\Xset$ of these
intersection points is essentially $(m{+}n{-}3)$-dependent.
\end{theorem}

We are going to consider a special type of $n$-poised sets defined by Chung and Yao:
\begin{definition}[\cite{CY77}]
An n-poised set $\Xset$ is called $GC_n$ set, if the $n$-fundamental polynomial of each node $A\in\Xset$ is a product of
$n$ linear factors.
\end{definition}

Now we are in a position to present the Gasca-Maeztu conjecture.
\begin{conjecture}[\cite{GM82}]
For any  $GC_n$ set $\Xset$ there is a maximal line, i.e., a line
passing through its $n+1$ nodes.
\end{conjecture}

Since now the Gasca-Maeztu conjecture was proved to be true only for $n \leq 5$. The case $n=2$ is trivial, and the case $n=3$ is easy to verify. The case
$n = 4$ first was proved by J. R.  Bush \cite{B90}. Several other proofs have been published since then (see e.g. \cite{CG01}, \cite{HJZ09b}, \cite{BHT}). For the case $n=5$ there is only one proof by H. Hakopian, K. Jetter and G. Zimmermann \cite{HJZ14}.

\subsection{The m-distribution sequence of a node}

In this section
we bring a number of concepts, properties and results from  \cite{HJZ14}.

Suppose that $\mathcal X$ is a $GC_n$ set. Consider a node $A\in\Xset$ together with the set of $n$ used lines $\Lset_A.$ 
The $N-1$ nodes of $\Xset\setminus\{A\}$ are somehow
distributed in the lines of $\Lset_A.$ 

Let us order the lines of $\Lset_A$ in the following way:

The line $\ell_1$ is a line in $\Lset_A$ that passes through maximal number of nodes of $\Xset,$ denoted by $k_1:$  $\Xset\cap\ell_1=k_1.$ 

The line $\ell_2$ is a line in $\Lset_A\setminus \{\ell_1\}$ that passes through maximal number of nodes of $\Xset\setminus\ell_1,$ denoted by $k_2:$  $(\Xset\setminus\ell_1)\cap\ell_2=k_2.$ 

In the general case the line $\ell_s,\ s=1,\ldots,n,$ is a line in $\Lset_A\setminus \{\ell_1,\ldots,\ell_{s-1}\}$ that passes through maximal number of nodes of the set $\Xset\setminus\cup_{i=1}^{s-1}\ell_i,$ denoted by $k_s:$  $(\Xset\setminus\cup_{i=1}^{s-1}\ell_i)\cap\ell_s=k_s.$

A correspondingly ordered  line sequence $$\Sset=(\ell_1,\ldots,\ell_n)$$ 
is called  a \emph{maximal line sequence} or briefly an \emph{m-line sequence}.
The sequence $(k_1,\ldots,k_n)$ is called a \emph{maximal distribution sequence}. Briefly
we call it \emph{m-distribution sequence} or \emph{m-d sequence.}

Evidently, for the m-d sequence we have that
\begin{equation}\label{nor}k_1 \geq k_2 \geq \cdots \geq k_n\ \hbox{and}\  k_1 + \cdots + k_n = N-1.\end{equation}

As it is shown in \cite{HJZ14} the m-distribution sequence for a node $A$ is unique, while it may correspond to
several m-line sequences.

Note that, an intersection point of several lines of $\Lset_A$
is counted for the line containing it which appears in $\Sset$ first.
Each node in $\Xset$ is called a \emph{primary} node for the line it
is counted for, and a \emph{secondary} node for the other lines
containing it.

According to Lemma~\ref{lm},
every used line has to contain at least two primary nodes, i.e.,
\begin{equation}\label{eq:kgeq2}
  k_i \geq 2 \quad\text{for } i=1,\ldots,n \,.
\end{equation}

\noindent Let  $\Sset=(\ell_1,\ldots,\ell_n)$ be an m-line sequence with the
associated m-d sequence $(k_1,\ldots,k_n)$\,.
\begin{lemma}[\cite{HJZ14}, Lemma 2.5]\label{lem:dist1}
Assume that $k_i=k_{i+1}=:k$ for some $i$. If the intersection point of
lines $\ell_i$ and $\ell_{i+1}$ belongs to $\Xset$, then it is a secondary
node for both $\ell_i$ and $\ell_{i+1}$. Moreover, interchanging
$\ell_i$ and $\ell_{i+1}$ in $\Sset$ still yields an m-line sequence.
\end{lemma}

We say that a
polynomial has $(s_i,\ldots,s_j)$ \emph{primary zeroes} in the lines
$(\ell_i,\ldots,\ell_j)$ if the zeroes are primary nodes in the
respective lines.
\smallskip
From Proposition \ref{pointsell} we get 
\begin{corollary}\label{cor22} If a 
polynomial $p\in\Pi_{m-1}$ has $(m,m-1,\ldots,m-k)$ primary zeroes in the lines
$(\ell_{m-k},\ell_{m-k+1}\ldots,\ell_{m})$ then we have that 
$
p = \ell_{m}\ell_{m-1}\cdots\ell_{m-k}  r  ,\ \text{where} \ r\in\Pi_{m-k}.
$
\end{corollary}

In some cases we shall fix a particular line $\wti\ell$ used by a node and then study
the properties of the other factors of the fundamental polynomial. In particular, this
will be the case for a line $\wti\ell$ that is shared by several nodes.

In this case in the corresponding m-line sequence, called $\wti\ell$-m-line sequence, we take as the first line $\ell_1$ the line $\wti\ell,$ no matter through how many nodes it passes. Then the second and subsequent lines are chosen, as in the case of the  m-line sequence. 

Thus the line $\ell_2$ is a line in $\Lset_A\setminus \{\wti\ell_1\}$ that passes through maximal number of nodes of $\Xset\setminus\wti\ell_1,$ and so on.

Correspondingly we define   $\wti\ell$-m-distribution  sequence.

\section{The Gasca-Maeztu conjecture for $n=5$}

Let us formulate the Gasca-Maeztu conjecture for
$n=5$ as:

\begin{theorem}\label{thmain2}
For any $GC_5$ set $\Xset$ of $21$ nodes there is a maximal line, i.e.,
a $6$-node line.
\end{theorem}
To prove the theorem assume by way of contradiction the following.

\begin{assumption}\label{B} The set $\Xset$ is a $GC_5$ set with no maximal line.
\end{assumption}
In view of  \eqref{nor} and
\eqref{eq:kgeq2} the only
possible m-d sequences for any node $A\in\Xset$ are
\begin{equation}\label{eq:5cases}(5,5,5,3,2);\quad (5,5,4,4,2);\quad (5,5,4,3,3);\quad (5,4,4,4,3);\quad (4,4,4,4,4).
\end{equation}

The results from \cite{HJZ14} below show how many times a line can be used, depending the number of nodes it passes through. In each statement it is assumed that $\Xset$ is a $GC_5$ set with no maximal line.
\begin{proposition}[\cite{HJZ14}, Prop.~2.11]\label{prp:ub1}
Suppose
that $\wti\ell$ is a $2$-node line.
Then $\wti\ell$ can be used by at most one node of $\Xset$.
\end{proposition}

\begin{proposition}[\cite{HJZ14}, Prop.~2.12]\label{prp:ub2}
Suppose
that $\wti\ell$ is a $3$-node line and is used by two nodes $A$, $B\in\Xset$. Then
there exists a
third node $C$ using $\wti\ell$\,. Furthermore, $A$, $B$, and $C$
share three other lines, each passing through five primary nodes. For
each of the three nodes, the m-d sequence is $(5,5,5,3,2)$,
and the other two nodes are the primary nodes in the respective fifth
line. In particular, $\wti\ell$ is used exactly three times.
\end{proposition}

\begin{proposition}[\cite{HJZ14}, Prop.~2.13]\label{prp:ub3}
Suppose
that a line $\wti\ell$ is used by three nodes $A$, $B$,
$C\in\Xset$. Then $\wti\ell$ passes through at least three nodes of
$\Xset$.
 
  \par
If $\wti\ell$ is a $4$-node line, then $A$, $B$, and
$C$ share $\wti\ell$ and three other lines, $\ell_2$ and $\ell_3$
passing through five and $\ell_4$ through four primary nodes. For each
of the three nodes, the \md-distribution sequence with respect to
$\wti\ell$ is $(4,5,5,4,2)$. $\wti\ell$ can only be used by $A$, $B$, and
$C$, i.e., it is used exactly three times.
\end{proposition}

\begin{corollary}[\cite{HJZ14}, Cor. 2.14]\label{cor:ub4}
Suppose
that a line $\wti\ell$ is used by four nodes in $\Xset$. Then
$\wti\ell$ is a $5$-node line.
\end{corollary}

\begin{proposition}[\cite{HJZ14}, Prop.~2.15]\label{prp:ub5}
Suppose
that a line $\wti\ell$ is used by five nodes in $\Xset$. Then
$\wti\ell$ is a $5$-node line, and it is actually
used by exactly six nodes in $\Xset$. These six nodes form a
$GC_2$ set and share two more lines with five primary nodes each, i.e.,
each of these six nodes has the m-d sequence $(5,5,5,3,2)$.
\end{proposition}

At the end we bring a (part of a) table from \cite{HJZ14} which follows from Propositions~\ref{prp:ub1},
\ref{prp:ub2}, \ref{prp:ub3}, Corollary~\ref{cor:ub4}, and
Proposition~\ref{prp:ub5}. It shows under which conditions a $k$-node line
$\wti\ell,\ 2\le k\le 5,$ can be used at most how often, provided that the considered $GC_5$ set has no maximal line.
\begin{equation}\label{eq:table}
  \begin{array}{c|c|c|}
    & \multicolumn{2}{c|}{\text{maximal \#\ of nodes
      using $\wti\ell$}} \\\hhline{~--}
    \text{total \#} & \text{in general} & \text{no node uses} 
      \\
    \text{of nodes}  & & (5,5,5,3,2)   \\
    \text{in $\wti\ell$} & \text{\phantom{no node uses}} &
      \text{m-d sequence}   \\
    &  &   \\\hline
    5 & 6 & 4  \\\hline
    4 & 3 & 3  \\\hline
    3 & 3 & 1 \\\hline
    2 & 1 & 1 \\\hline
  \end{array}
\end{equation}

\subsection{The case
$(5\,{,}\,5\,{,}\,5\,{,}\,3\,{,}\,2)$}\label{ssec:case55532}
\vspace{0.5cm}
In this and the following sections, we will prove the following 
\begin{proposition}\label{thm:main3} Assume that $\Xset$ is a $GC_5$ set with no maximal line. Then for no node in $\Xset$ the m-d sequence is $(5,5,5,3,2)$.
\end{proposition}
Assume by way of contradiction the following.
\begin{assumption}\label{C} 
$\Xset$ contains a node for
which an m-line sequence $(\ell_1,\ell_2,\ell_3,\ell_4,\ell_5)$ implies the m-d sequence $(5,5,5,3,2)$.
\end{assumption}
Set  $\Xset=\Aset \cup \Bset$ (see Fig. \ref{case1}), with
\begin{equation*}
  \Aset = \Xset \cap \{\ell_1 \cup \ell_2 \cup \ell_3 \}, \quad
  \#\Aset = 15, \quad \text{and} \quad
  \Bset = \Xset \setminus \Aset, \quad  \#\Bset = 6.
\end{equation*}

Denote $\Lset_3:=\{\ell_1,\ell_2,\ell_3\}.$
Note that no intersection point of the three lines of $\Lset_3$ belongs to $\Xset$.

Below we bring a simple proof for
\begin{lemma}[\cite{HJZ14}, Lemma 3.2]\label{lem:AL}
\nix\vspace{-2mm}
\begin{enumerate}
\setlength{\itemsep}{0mm}
\item
The set $\Bset$ is a $GC_2$ set, and each node $B\in\Bset$ uses the three lines of $\Lset_3$ and the two lines it uses within $\Bset$, i.e.,
\begin{equation}\label{ban}
  p_{B,\Xset}^\star = \ell_1 \, \ell_2 \, \ell_3 \, p_{B,\Bset}^\star
  \,.
\end{equation}
\item
No node in $\Aset$ uses any of the lines of $\Lset_3.$

\end{enumerate}
\end{lemma}
\begin{figure}
\begin{center}
\includegraphics[width=10.0cm,height=5.cm]{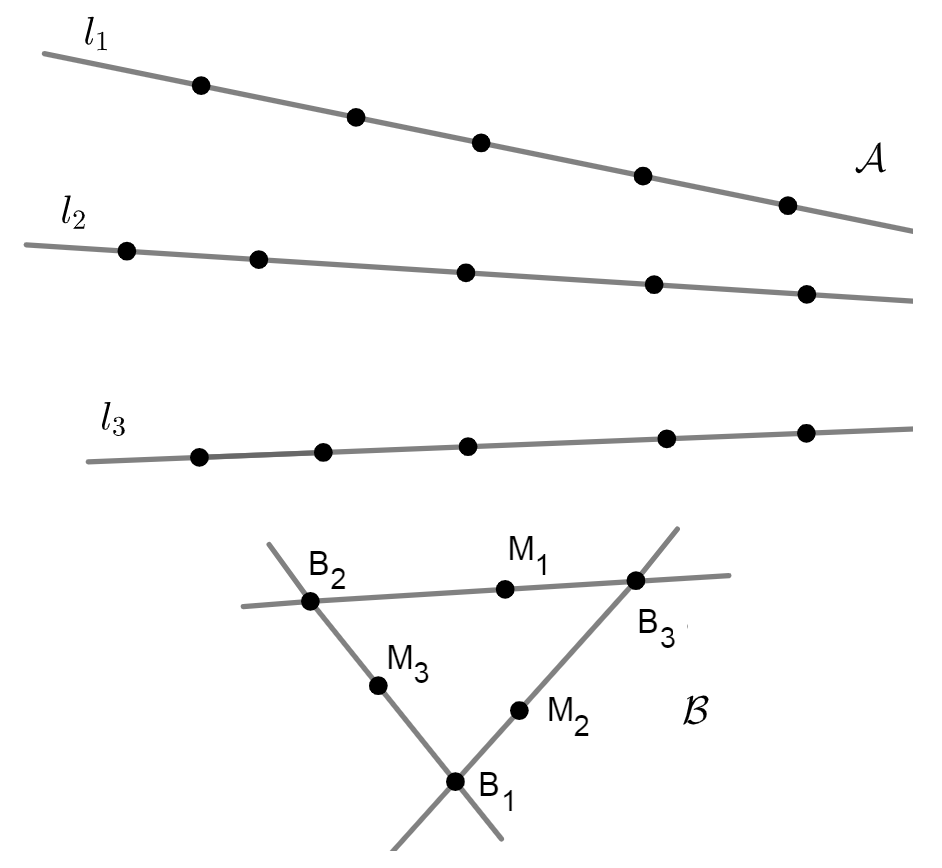}
\end{center}
\caption{The case 
$(5,5,5,3,2)$ with $\Xset=\Aset\cup\Bset.$} \label{case1}
\end{figure}

\begin{proof}
(i) Suppose by way of contradiction that the set $\Bset$ is not $2$-poised, i.e., it is a subset of a conic $\Cset.$ Then $\Xset$ is a subset of the zero set of the polynomial $\ell_1 \, \ell_2 \, \ell_3 \, \Cset,$ which contradicts Proposition \ref{poisedii}. Then we readily obtain the formula \eqref{ban}. 

(ii) Without loss of generality assume that $A \in \ell_1$ uses the line $\ell_2.$ Then $p_{A}^\star = \ell_2\,q,$ where $q \in \Pi_4.$ It is easily seen that $q$ has (5,4) primary zeros in the lines $(\ell_3 , \ell_1).$ Therefore, in view of Corollary \ref{cor22}, we obtain that $p_{A}^\star = \ell_2\,\ell_3\,\ell_1\,r,$ which is a contradiction.
 \par

\end{proof}

Evidently, any node in a $GC_2$ set uses a maximal line, i.e., $3$-node line. Hence we conclude readily that any $GC_2$ set, including also $\Bset,$ possesses  at least three maximal lines (see Figure \ref{case1}).

A node $A\in\Xset$ is called a $2_m$-node if it is the intersection point of two maximal lines. Note that the nodes $B_i,\ i=1,2,3,$ in Fig. \ref{case1}, are $2_m$-nodes for $\Bset.$

\begin{definition} We say, that a line $\ell$ is a $k_\Aset$-node line if it passes through exactly $k$ nodes of $\Aset.$\end{definition}

\begin{lemma}\label{lem01}
(i) Assume that a line $\wti\ell\notin\Lset_3$  does not intersect  a line $\ell\in\Lset_3$ at a node in $\Xset.$ Then the line 
$\wti\ell$ can be used at most by one node from $\Aset.$ Moreover, this latter node  belongs to $\ell\cap\Aset.$ \par
(ii) If a line $\ell$ is $0_{\Aset}$ or $1_{\Aset}$-node line then no node from $\Aset$ uses the line $\ell.$ 
\par
(iii) If a line $\ell$ is $2_{\Aset}$-node line then $\ell$ can be used by at most one node from $\Aset.$ \par
(iv) Suppose $\ell$ is a maximal line in $\Bset.$ Then $\ell$ can be used by at most one node from $\Aset.$
\end{lemma}
\begin{proof}
(i) Without loss of generality assume that $\ell=\ell_1$ and $A \in \ell_2 $ uses $\wti\ell:$ 
\begin{equation*}
    p_{A}^\star =\wti\ell\,q,\ q \in \Pi_4.
\end{equation*}
It is easily seen that $q$ has $(5,4,3)$ primary zeros in the lines $(\ell_1 , \ell_3 , \ell_2).$ Therefore, in view of Corollary \ref{cor22}, we conclude that 
$p_{A}^\star = \wti\ell\,\ell_1\,\ell_2\,\ell_3\,r,\ r\in \Pi_1,$ 
which is a contradiction.

Now assume conversely that $A,B \in \ell_1\cap\Xset$ use the line $\wti\ell.$ Choose a point $C \in \ell_2\setminus(\wti\ell\cup\Xset).$ Then choose numbers $\alpha$ and $\beta,$ with $|\alpha| + |\beta| \neq 0,$ such that $p(C)=0,$ where $p:=\alpha p_A^\star + \beta p_B^\star.$ It is easily seen that $p =\wti\ell\,q,\ q \in \Pi_4$ and  the polynomial $q$ has $(5,4,3)$ primary zeros in the lines $(\ell_2,\ell_3,\ell_1).$ Therefore $p = \wti\ell\ell_1\,\ell_2\,\ell_3\,q,$ where $q\in\Pi_1.$ Thus $p(A) = p(B) = 0,$ implying that $\alpha=\beta = 0$, which is a contradiction.

\par
The items (ii) and (iii) readily follow from (i).
The item (iv) readily  follows from (iii).
\end{proof}
\noindent Denote by $\ell_{AB}$ the line passing through the points $A$ and $B.$
\begin{proposition}\label{Prop.14}
Let $\ell_{B_1M_1}$ be 5-node line, which is used by all the six nodes of a subset $\Aset_6\subset\Aset.$ Suppose also that  $\ell$ is a 4-node line passing through $B_1$. If the line $\ell$ is used by three nodes from $\Aset$ then all these three nodes belong to $\Aset_6.$
\end{proposition}

\begin{proof} The six nodes of $\Aset_6$ use the $5$-node line $\ell_{B_1M_1}.$ Therefore, in view of Proposition \ref{prp:ub5}, these six nodes share also two more lines passing through five primary nodes. It is easily seen that these latter two lines are the lines $\ell_{B_2M_2}$ and $\ell_{B_3M_3}.$ Assume by way of contradiction that the nodes $D_1, D_2, D_3\in\Aset$ are using the line $\ell$ and $D_1\notin\Aset_6.$ According to Proposition \ref{prp:ub3} these three nodes share also two lines passing through five primary nodes. In view of Lemma \ref{lem01}, (iv), these latter two lines cannot be maximal lines in $\Bset.$ Therefore they belong to the set $\{\ell_{B_2M_2} ,  \ell_{B_3M_3} , \ell_{M_1M_2} , \ell_{M_2M_3} , \ell_{M_1M_3}\}$. One of them should be $\ell_{B_2M_2}$ or $\ell_{B_3M_3}$, since any two lines from $\{\ell_{M_1M_2} , \ell_{M_2M_3} , \ell_{M_1M_3}\}$ share a node. Therefore one of them will be used by seven nodes, namely by $D_1$ and the nodes of $\Aset_6.$  This contradicts Proposition \ref{prp:ub5}.
\end{proof}

 \subsection{The proof of Proposition \ref{thm:main3}}

Consider all the lines passing through $B:=B_1$ and at least one more node of ${\mathcal X}$. Denote the set of these lines by $\Lset(B).$  Let $m_k(B),\ k=1,2,3,$ be the number of $k_{\Aset}$-node lines from $\Lset(B).$ 

We have that \begin{equation}\label{eq14}
  1m_1(B) + 2m_2(B) + 3m_3(B)   =
  \# \Aset = 15 .
\end{equation}
\begin{lemma}\label{lem05}
Suppose that a line $\ell,$ passing through $B$ and different from the line $\ell_{BM_1},$ is  a $3_{\Aset}$-node line. Then $\ell$ can be used by at most three nodes from $\Aset.$

\end{lemma}
\begin{proof} Note that $\ell$ is not a maximal line for $\Bset,$ since otherwise $\ell$ will be a maximal line for $\Xset.$ Therefore $\ell$ is a $4$-node line and Proposition \ref{prp:ub3} completes the proof.
\end{proof}
 \begin{lemma}\label{lem99} We have that
 $m_3(B) \leq 4.$
 \end{lemma}
 \begin{proof}
 The equality \eqref{eq14} implies that $m_3(B) \leq 5.$ Assume by way of contradiction that five lines pass through $B$ and three nodes in $\Aset.$
Therefore these five lines intersect the three lines $\ell_1, \ell_2, \ell_3,$ at the $15$ nodes of $\Aset.$ Then, by Theorem \ref{thm:C-B}, these $15$ nodes are $5 + 3 - 3 = 5$-dependent, which is a contradiction.
 \end{proof}
 
Now we are in a position to start

\begin{proof}[Proof of Proposition \ref{thm:main3}] In view of Proposition \ref{prp:ub5} we divide the proof into the following three cases.
\hfill\break
\emph{Case 1.} Suppose that $\ell_{BM_1}$ is 5-node line used by six nodes from $\Aset.$ 

Denote the set of these six nodes by $\Aset_6\subset \Aset.$

We have that any node from $\Aset$ uses at least one line from $\Lset(B).$ Proposition \ref{Prop.14} implies that all  $3_{\Aset}$-node lines  from $\Lset(B),$ except $\ell_{BM_1},$ can be used by at most two nodes from $\Aset\setminus\Aset_6.$ 

From Lemma \ref{lem01}, we have that
 \begin{equation}
     15-6\leq  0m_1(B) + 1m_2(B)  + 2(m_3(B)-1).
 \end{equation}
 In view of \eqref{eq14} we get
 \begin{equation}\label{eq102}
    m_1(B) + 2m_2(B) + 3m_3(B)-6\leq 1m_2(B) + 2m_3(B) -2. 
 \end{equation}
Therefore we conclude that $m_1(B) + m_2(B) + m_3(B) \leq 4,$ or, in other words,  $3m_1(B) + 3m_2(B) +3m_3(B) \leq 12,$ which contradicts equality \ref{eq14}.\\

 \emph{Case 2.} Suppose that $\ell_{BM_1}$ is not $5$-node line.
 
Then, in view of the table \eqref{eq:table}, it can be used by at most three nodes of $\Aset.$ From Lemmas \ref{lem01} and \ref{lem05}, (ii),(iii), we have that
\begin{equation}\label{eq100}
    15\leq 1m_2(B) + 3m_3(B).
\end{equation}
In view of \eqref{eq14} we get
\begin{equation}\label{eq101}
    m_1(B) + 2m_2(B) + 3m_3(B) \leq m_2(B) + 3m_3(B).
\end{equation}
Hence $m_1(B) = m_2(B) = 0$ and $m_3(B) \ge 5,$ which contradicts Lemma \ref{lem99}. \\

\emph{Case 3.} Suppose that $\ell_{BM_1}$ is 5-node line used by  at most four nodes of $\Aset.$ 

In this case we have that  
\begin{equation*}
    15 \leq 1m_2(B)  + 3(m_3(B) - 1) + 4.
\end{equation*}
In view of \eqref{eq14} we get
\begin{equation}
    m_1(B) + 2m_2(B) + 3m_3(B) \leq 1m_2(B)  + 3m_3(B) + 1.
\end{equation} Hence $2m_1(B) + 2m_2(B) \leq 2.$ In view of \eqref{eq14} we have that
\begin{equation}
    3m_3(B_1) \geq 13, 
\end{equation}
which contradicts Lemma \ref{lem99}.
\end{proof}

 \subsection{The cases $(5\,{,}\,5\,{,}\,4\,{,}\,4\,{,}\,2)$ , $(5\,{,}\,5\,{,}\,4\,{,}\,3\,{,}\,3)$ and $(5\,{,}\,4\,{,}\,4\,{,}\,4\,{,}\,3)$}

Let us fix a node $A\in\Xset$ and consider the set of lines $\Lset(A).$ Let $n_k(A)$ be the number of $(k + 1 )$-node lines from $\Lset_A.$ In view of Assumption \ref{B} we have that
\begin{equation}\label{eq25}
    1n_1(A) + 2n_2(A) + 3n_3(A) + 4n_4(A)  =
  \# \bigl({\mathcal X}\setminus\{A\}\bigr) = 20 .
\end{equation}

Next we bring a result from \cite{HJZ14}. We present also the proof for the convenience.
\begin{lemma}[\cite{HJZ14}, Lemma 3.13]\label{lem20}
Assume that $\Xset$ is a $GC_5$ set with no maximal line. By
Proposition~\ref{thm:main3}, for no node of $\Xset$ the
m-d sequence is $(5,5,5,3,2)$.  Then the following hold.
\begin{enumerate}
\setlength{\itemsep}{0mm}
\item
There is no $3$-node line and $m$-node line is used exactly $m-1$ times, where $m=2, 4, 5.$ 
\item
No two lines used by the same node intersect at a node in
$\Xset.$ 
\end{enumerate}
\end{lemma}

\begin{proof}
(i) Consider all the lines in $\Lset(A)$. From the third column of the table in~\eqref{eq:table}, it follows that for the total number $M(A)$ of uses of these lines, we have that
\begin{equation}\label{eq:20uses}
  M(A) \leq 1\,n_1(A) + 1\,n_2(A) + 3\,n_3(A) + 4\,n_4(A) \,.
\end{equation}
Since each node in $\Xset\setminus\{A\}$ uses at least one line through
$A$, we must have $M(A)\geq 20$. In view of the equality \eqref{eq25} we conclude that  $M(A)=20$ and $n_3(A)=0$.
  \par
Moreover, we deduce that any  line containing $m$ nodes
including $A$ has to be used exactly $m{-}1$ times, where $m=2, 4, 5$. Since the node $A$ is arbitrary, this is true for all
lines containing at least two nodes of $\Xset$.
  \par\smallskip
(ii) Assume conversely that two lines $\ell_1,\ell_2,$ used by a node $A\in\Xset$
intersect at a node $B\in\Xset$. Then each of the nodes in
$\Xset\setminus\{A,B\}$ uses at least one line through $B$, while the node $A$ uses at least two lines. Thus we have $M(A)\geq21,$ which is a contradiction.
\end{proof}
\begin{corollary}\label{Cor.20}
For no node in $\Xset$ the m-d sequence is $(5,5,4,3,3)$ or $(5,4,4,4,3)$.
\end{corollary}
\begin{proof}
Suppose, that for a node $A \in \Xset,$ the m-d sequence is $(5,5,4,3,3)$ or $(5,4,4,4,3)$. In view of Lemma \ref{lem20}, (ii), there are no secondary nodes in the used lines. Thus the presence of $3$ the m-d sequence implies presence of a $3$-node line in an $m$-line sequence, which contradicts  Lemma \ref{lem20}, (i).
\end{proof}

\begin{proposition}\label{ve}
For no node in $\Xset$ the m-d sequence is $(5,5,4,4,2)$.
\end{proposition}
\begin{proof}
Assume that for a node $A\in\Xset$ some m-line sequence $(\ell_1,\ell_2,\ell_3,\ell_4,\ell_5)$ implies the m-d sequence $(5,5,4,4,2)$. In view of Lemma \ref{lem20}, (ii), the lines $\ell_1,...,\ell_5,$ contain exactly ${5, 5, 4, 4, 2,}$ nodes, respectively. Denote by $B$ and $C$ the two nodes in the line $\ell_5.$ Then we have 
\begin{equation*}
  p_B^\star= \ell_1\,\ell_2\,\ell_3\,\ell_4\,\ell_{AC}
  \quad\text{and}\quad
  p_C^\star= \ell_1\,\ell_2\,\ell_3\,\ell_4\,\ell_{AB}).
\end{equation*}
In view of Lemma \ref{lem20} the line $\ell_1$ is used by exactly four nodes of $\Xset$. Therefore, there exists a node $D\in\Xset\setminus\{A,B,C\},$ which is using the line $\ell_1.$ 
  
  In view of  \eqref{eq:5cases}, Proposition \ref{thm:main3}, and Corollary \ref{Cor.20}, for the node $D\in\Xset$ some m-line sequence $(\ell_1,\ell_2',\ell_3',\ell_4',\ell_5')$ yields the m-d sequence $(5,5,4,4,2)$.
  
  Now, as above, we have that the two nodes in the line $\ell_5'$ use the line $\ell_1.$ In view of Proposition \ref{prp:ub1}, the line $\ell_5',$ used by the node $D,$ cannot coincide with the lines $\ell_{AB}, \ell_{AC}$ or $\ell_{BC}$. Therefore $\ell_5'$ contains a node different from $A,B,C,D.$ Hence, the line $\ell_1$ is used at least five times, which is a contradiction. 
\end{proof}

\subsection{Proof of theorem \ref{thmain2}}

What is left to complete the proof of Theorem \ref{thmain2} is the following
\begin{proposition}\label{Cor.21}
For no node in $\Xset$ the m-d sequence is $(4,4,4,4,4)$.
\end{proposition}

\begin{proof}  Let us fix a node $A\in\Xset$. In view of  \eqref{eq:5cases}, Propositions \ref{thm:main3}, \ref{ve} and Corollary \ref{Cor.20}, for the node $A$, m-d sequence is $(4,4,4,4,4)$. Thus, in view of Lemma \ref{lem20}, (ii), all used lines are $4$-node lines. Therefore, in view of Lemma \ref{lem20}, (i), we conclude that $n_1(A)=n_2(A)=n_3(A)=n_4(A)=0.$
Now, the equality  \eqref{eq25} implies that $3n_3(A) = 20$, which is not possible.
\end{proof}



\begin{thebibliography}{99}
 
\bibitem{BHT}
V. Bayramyan, H. Hakopian, and S. Toroyan, A simple proof of the Gasca-Maeztu
conjecture for $n = 4,$ Jaen J. Approx. {\bf 7} (1) (2015), 137–147.

 

\bibitem{dB07} 
C.~de Boor, {Multivariate polynomial interpolation: conjectures
concerning GC sets}, Numer.\ Algorithms {\bf 45} (2007) 113--125.


\bibitem{B90}
J.~R.~Busch,
{A note on Lagrange interpolation in $\mathbb{R}^2$},
Rev.\ Un.\ Mat.\ Argentina {\bf 36} (1990) 33--38.


\bibitem{CG01}
J.~M.~Carnicer and M.~Gasca,
{A conjecture on multivariate polynomial interpolation},
Rev.\ R.~Acad.\ Cienc.\ Exactas F{\'i}s.\ Nat.\ (Esp.), Ser.~A
Mat.\ {\bf 95} (2001) 145--153.

\bibitem{CY77}
K.~C.~Chung and T.~H.~Yao,
{On lattices admitting unique Lagrange interpolations},
SIAM J.\ Numer.\ Anal.\ {\bf 14} (1977) 735--743.

\bibitem{E96}
{D.~Eisenbud, M.~Green, and J.~Harris }
{Cayley-Bacharach theorems and conjectures}, Bull. Amer. Math. Soc.
(N.S.), {\bf 33} (3), (1996) 295--324.


\bibitem{GM82}
M.~Gasca and J.~I.~Maeztu,
{On Lagrange and Hermite interpolation in $\mathbb{R}^k$},
Numer.\ Math.\ {\bf 39} (1982) 1--14.

\bibitem{HJZ09a}
H.~Hakopian, K.~Jetter, and G.~Zimmermann,
{Vandermonde matrices for intersection points of curves},
Ja\'en J.\ Approx.\ {\bf 1} (2009) 67--81.

\bibitem{HJZ09b}
H.~Hakopian, K.~Jetter, and G.~Zimmermann,
{A new proof of the Gasca-Maeztu conjecture for $n=4$},
J.\ Approx.\ Theory {\bf 159} (2009) 224--242.

\bibitem{HJZ14}
H. Hakopian, K. Jetter, and G. Zimmermann, The Gasca-Maeztu
conjecture for $n=5$,  Numer. Math.,  {\bf 127} (2014) 685--713.

\end{thebibliography}
\end{document}